\documentclass[a4paper,12pt,oneside]{article}
\usepackage{amsthm,amsmath,amsfonts,amssymb}
\usepackage{color}
\AtEndDocument{\bigskip{\footnotesize%
{\textsc{Department of Mathematical Sciences, University of Essex, Colchester, Essex CO4 3SQ, U.K.}} \par  
  \hspace{-0.7cm}\textit{E-mail address}, Ihechukwu Chinyere: \texttt{ihechukwu.chinyere@essex.ac.uk}  \par
  \addvspace{\medskipamount}
 \hspace{-0.7cm}\textsc{Department of Mathematics, Kwame Nkrumah University of Science and Technology, Ghana.} \par  
 \hspace{-0.7cm} \textit{E-mail address}, Bernard Oduoku Bainson: \texttt{bboduoku@knust.edu.gh} 
}}


\newtheorem{theorem}{Theorem}
\newtheorem{lem}[theorem]{Lemma}
\newtheorem{cor}[theorem]{Corollary}
\newtheorem{conj}[theorem]{Conjecture}
\newtheorem{quest}[theorem]{Question}

\newtheorem{defn}[theorem]{Definition}

\newtheorem{theoremx}{Theorem}
\newtheorem{corx}[theoremx]{Corollary}

\numberwithin{theorem}{section}

\begin{document}
\title{Perfect Prishchepov groups}
\author{Ihechukwu Chinyere\thanks{Corresponding author} \ and Bernard Oduoku Bainson}

\maketitle

\begin{abstract}
We study cyclically presented groups of type $\mathfrak{F}$ to determine when they are perfect. It turns out that to do so, it is enough to consider the Prishchepov groups, so modulo a certain conjecture, we classify the perfect Prishchepov groups $P(r,n,k,s,q)$ in terms of the defining integer parameters $r,n,k,s,q$. In particular, we obtain a classification of the perfect Campbell and Robertson's Fibonacci-type groups $H(r,n,s)$, thereby proving a conjecture of Williams,  and yielding a complete classification of the groups $H(r,n,s)$ that are connected Labelled Oriented Graph groups. 
\end{abstract}

\noindent \textbf{Keywords:} Prishchepov groups, cyclically presented groups, perfect groups, LOG groups.\newline
\noindent \textbf{MSC:} $12\text{E}05, 20\text{F}05, 20\text{F}65, 57\text{M}07, 57\text{M}15$.

\section{Introduction}\label{sec:Introdution}

Let $F_n$ be the free group of rank $n$ with basis $X=\lbrace x_0,x_1,\cdots,x_{n-1}\rbrace$. Let $\theta$ denote the shift automorphism of $F_n$ of order $n$ such that $x_i\mapsto x_{i+1} ~(0\leq i<n)$, where the subscripts are taken modulo $n$. A group presentation is called a \emph{cyclic presentation} if it is of the form	
\begin{equation}\label{eqn:Cyclicpresentation}
P_n(w)=\langle x_0,x_1,\cdots,x_{n-1}~|~\ w,\theta (w), \cdots , \theta^{n-1} (w)\rangle
\end{equation}
for some positive integer $n$, where $w$ is a (cyclically) reduced word on  $X^{\pm 1}=\lbrace x_0^{\pm 1},x_1^{\pm 1},\cdots,x_{n-1}^{\pm 1}\rbrace$ (see \cite{johnson1974fibonacci}). A group  is said to be \emph{cyclically presented} if it has a presentation of the form  (\ref{eqn:Cyclicpresentation}), in which case the group is often denoted by $G_n(w)$. 
The shift defines an automorphism of
$G_n(w)$ with exponent $n$ and the resulting $\mathbb{Z}_n$-action on $G_n(w)$ determines the
shift extension $E_n(w)= G_n(w)\rtimes_{\theta}\mathbb{Z}_n,$ which admits a 
balanced presentation of the form
\begin{equation*}
E_n(w)= \langle x,t~|~t^n, W\rangle,
\end{equation*}
where $W=W(x, t)$ is obtained by rewriting $w$ in terms of the substitutions
$x_i=t^ixt^{-i}$ (see \cite[Theorem 4]{johnson1974fibonacci}), and the shift on $G_n(w)$ is then
realized by conjugation by $t$ in $E_n(w)$. The group $G_n(w)$ is recovered as the
kernel of the retraction $\nu^0: E_n(w) \longrightarrow \mathbb{Z}_n=\langle t~|~t^n\rangle$, such that $x$ maps to the identity.

\medskip
Cyclically presented groups have been an object of various investigations not just from the algebraic point of view as in \cite{bogley2017coherence,campbell1975metacyclic,cavicchioli2008some,chinyere2020hyperbolicity,
chinyere2020hyperbolic,edjvet2017infinite,williams2012largeness} for example, but also from topological perspectives due to their close connections with the topology of closed orientable $3$-manifolds (see for example, \cite{cannon2016bitwist, cavicchioli1998geometric,cavicchioli1996manifold,  cavicchioli2005families,howie2017fibonacci,sieradski1986combinatorial,telloni2010combinatorics}). This article  takes the former viewpoint and concerns the family of cyclically presented groups of type $\mathfrak{F}$ from \cite{bogley2017coherence}, which are defined by presentations
\begin{equation}\label{eqn:Prischepovpresentation}
\mathcal{P}=\Bigg\langle x_0, \cdots , x_{n-1} ~\Bigg |~\prod^{r-1}_{i=0}x_{j+iq}\Bigg(\prod^{s-1}_{i=0}x_{j+k-1+iq}\Bigg)^{\epsilon} ~(0\leq j<n) \Bigg\rangle,
\end{equation}
where $k,q,r,s\geq 1$ and $n\geq 2$ are integers, $\epsilon\in \lbrace\pm 1\rbrace$ and subscripts taken modulo $n$. Our aim is to determine when cyclically presented groups of type $\mathfrak{F}$ are perfect or trivial. According to \cite[Lemma 1]{bogley2017coherence}, a group with presentation $\mathcal{P}$ is the
kernel of the retraction \[\nu^q: E=\langle y,t~|~t^n, y^rt^Ay^{\epsilon s}t^{-B}\rangle \longrightarrow \mathbb{Z}_n=\langle t~|~t^n\rangle,\] such that $\nu^q(x)=t^q$ and $\nu^q(t)=t$; and the parameters $r,n,k,s,q,\epsilon,A,B$ are related by the relations  $q(r+\epsilon s)\equiv (B - A)\bmod n$ and 
\begin{equation*}
k\equiv \left\{
  \begin{array}{@{}ll@{}}
    (qr+A+1)\bmod n, & \text{if}\ \epsilon=1 \\
    (q(r-s)+A+1)\bmod n, & \text{if} \ \epsilon=-1\\
  \end{array}\right.
\end{equation*}
where the second relation comes from rewriting as in \cite{bogley2017coherence,prishchepov1995aspherisity, williams2012largeness}. The Prishchepov groups $P(r,n,k,s,q)$ are of principal interest in this article and they arise as the groups defined by $\mathcal{P}$ in the special case where $\epsilon=-1$, which means that 
\begin{align*}
A&\equiv (k-q(r-s)-1)\bmod n, ~\text{and}\\
B&\equiv (k-1)\bmod n.
\end{align*}
As the name implies the  Prishchepov groups were first considered in \cite{prishchepov1995aspherisity} by Prishchepov; and later by others (see for example \cite{cavicchioli2003topological,Fulvia2006,telloni2010combinatorics,williams2012largeness}, noting that the roles of $A$ and $B$ are interchanged in \cite{Fulvia2006,williams2012largeness}).  For particular choices of parameters these groups are known: Conway’s Fibonacci groups $F(2,n)=P(2, n, 3, 1, 1)$ of \cite{conway1965advanced},  the groups of Fibonacci type $G_n (m, k) = P (2, n, k+1 , 1, m)$ of  \cite{cavicchioli1996manifold,johnson1975some}  and more (see \cite{campbell1975metacyclic,campbell1975note,campbell1975class,
campbell1979finitely,gilbert1995log,sieradski1986combinatorial}). Therefore the Prishchepov groups provide a means to obtain a generalized treatment of these groups. More important to us in this article is that to classify the perfect cyclically presented groups of type $\mathfrak{F}$, it is enough to do so for the Prishchepov groups (see Lemma \ref{lem:r-s}), hence the rest of this work focuses on the groups $P(r,n,k,s,q)$. 

\medskip
In describing the properties of $P(r,n,k,s,q)$ in terms of the parameters $k,n,r,s,q$ it is often helpful to make the following assumptions. Firstly, we can assume that $r\geq s$ since the groups $P(r,n,k,s,q)$ and $P(s,n,n-k+2,r,q)$ are isomorphic by \cite[Lemma 14]{williams2012largeness}. For the second let $N=n/(n,q,k-1)$, $Q=q/(n,q,k-1)$ and $K-1\equiv ((k-1)/(n,q,k-1))\bmod N$. Then as described in \cite{MR1961572} (see also, \cite[Corollary 3]{williams2012largeness}) the group $P(r,n,k,s,q)$ decomposes as a free product of $(n,q,k-1)$ copies of $P(r,N,K,s,Q)$. In particular $P(r,n,k,s,q)$ is perfect (resp. trivial) if and only if  $P(r,N,K,s,Q)$ is perfect (resp. trivial). Hence, when suitable we may assume that $(n,q,k-1)=1$, which is the same as saying that $\mathcal{P}$ as given in (\ref{eqn:Prischepovpresentation}) is \emph{irreducible}. We may further assume that $q=1$,  since $P(r,n,k,s,q)$ being perfect and $(n,q,k-1)=1$ implies $(n,q)=1$ (see Lemma \ref{lem:gcd(n,q)dividesk-1}), in which case $P(r,n,k,s,q)\cong P(r,n,\hat{q}(k-1)+1,s,1)$ where $q\hat{q}\equiv 1\bmod n$ by \cite[Theorem 16]{williams2012largeness}. The results in the general setting can then be obtained as corollaries of the ones with the above assumptions.

\medskip
Two subfamilies of $P(r,n,k,s,q)$ play a vital role in this article for completely opposite reasons as we shall see. These are the groups $H(r, n, s) = P(r, n, r + 1, s, 1)$ of \cite{campbell1975class}, and the  generalized Sieradski groups $S(r, n) = P(r, n, 2, r -1, 2), ~ (r \geq  2)$ of \cite{cavicchioli1998geometric}. In fact, the main motivation for attempting to classify the perfect cyclically presented groups of type $\mathfrak{F}$ is the following conjecture by  Williams  regarding the perfectness of the groups $H(r,n,s)$. 

\begin{conj}\cite[Conjecture 1]{williams2019generalized}\label{conj:WilliamsConjecture}
	Let  $r,s\geq 1$ and $n\geq 2$ such that $r\not\equiv 0\bmod n$, $s\not\equiv 0\bmod n$. Then $H(r,n,s)^{ab}\neq 1$. 
\end{conj}
Our main result completely describes when the group $H(r,n,s)$ is perfect or trivial in terms of the parameters $r,n,s$. As an immediate consequence, we obtain proof of Conjecture \ref{conj:WilliamsConjecture}. The precise statement of our main result is the following.
\begin{theoremx}\label{theorem:MainResult1}
For integers $r,s\geq 1$ and $n\geq 2$, the following are equivalent:
\begin{enumerate}
\item The group $H(r,n,s)$  is perfect;
\item The group $H(r,n,s)$ is trivial;
\item  $|r-s|=1$ and either $r\equiv 0\bmod n$ or $s\equiv 0\bmod n$.
\end{enumerate}
\end{theoremx}
In \cite{williams2019generalized} an incomplete classification of the groups $H(r,n,s)$  that are connected Labelled Oriented Graph (LOG) groups (see \cite{gilbert2017labelled,gilbert1995log,howie1985asphericity,howie2012tadpole,williams2019generalized} for more details about LOG groups) was given, the missing ingredient being a proof of Conjecture \ref{conj:WilliamsConjecture}. Therefore, with the conjecture now proved we have a complete classification result which can be stated as follows.

\begin{corx}\label{theorem:LOGgroup}
	Let $r,s\geq 1$  and $n\geq 2$. Then the group $H(r,n,s)$ is a connected LOG group if and only if one of the following holds:
	\begin{enumerate}
		\item $r=s$ and $(n,r)=1$, in which case $H(r,n,s)$ is the $(n,r)$-torus knot group.
		\item $(n,r,s)=2$ and either $r\equiv 0\bmod n$ or $s\equiv 0\bmod n$, in which case $H(r,n,s)\cong\mathbb{Z}$.
	\end{enumerate}
\end{corx}	
Next, we turn attention to the Prishchepov groups of type $\tilde{\mathfrak{Z}}$.
\begin{defn}[type $\tilde{\mathfrak{Z}}$] The group  $P(r,n,k,s,q)$ is of \emph{type $\tilde{\mathfrak{Z}}$} if any of the following holds:
\begin{align*}
\emph{type}~ \mathfrak{Z} &: q(r-s)\equiv 2(k-1)\bmod n ~(\text{eqivalently}~ A\equiv B\bmod n); or\\
\emph{type}~ \mathfrak{Z}'&: q(r+s)\equiv 0\bmod n.
\end{align*}
\end{defn}
The type \emph{$\mathfrak{Z}$} was introduced in \cite{mcdermott2017topological} and notably contains the generalized Sieradski groups. The two types are related in the sense that  $P(r,n,k,r-1,1)$ is of type $\mathfrak{Z}$ (resp. type $\mathfrak{Z}'$) if and only if $P(k-1,n,r+1,k-2,1)$ is of type $\mathfrak{Z}'$ (resp. $\mathfrak{Z}$), and Lemma \ref{lem:isoofreppoly} shows that the abelianizations of the two groups are related. Also, if $(n,q)=1$ and $P(r,n,k,s,q)$ is both of type $\mathfrak{Z}$ and of {type} $\mathfrak{Z}'$, then it can be shown that either $n$ is odd and $P(r,n,k,s,q)\cong H(r,n,s)$ or $n$ is even and $P(r,n,k,s,q)$ maps onto $H(r,n/2,s)$. The latter observation together with Theorem \ref{theorem:MainResult1}, is crucial in describing the perfect groups $P(r,n,k,s,q)$ of type $\tilde{\mathfrak{Z}}$. Indeed, part (2) of Theorem \ref{theorem:MainResult2} essentially excludes the groups that are both of type $\mathfrak{Z}$ and of type $\mathfrak{Z}'$.

 \begin{theoremx}\label{theorem:MainResult2}
Let $n\geq 2$, $k,q\geq 1$ and $r\geq s\geq 1$ be integers such that $(n,k-1,q)=1$. Suppose $P(r,n,k,s,q)$ of type $\tilde{\mathfrak{Z}}$. Then the following are equivalent:
\begin{enumerate}
\item The group $P(r,n,k,s,q)$ is perfect;
\item $|r-s|=1$, $(n,q)=1$ and $(k-1-qr,n)=1$.
\end{enumerate}
\end{theoremx}
The classification of the trivial cyclically presented groups of type $\mathfrak{Z}$ easily follows from Theorem \ref{theorem:MainResult2}, but can also be deduced from \cite[Theorem 20] {williams2012largeness}.

\begin{corx}\label{cor:ToMainResult2}
Let $n\geq 2$, $k,q\geq 1$ and $r \geq s\geq 1$ be integers such that $(n,k-1,q)=1$. Suppose $P(r,n,k,s,q)$ is of type $\mathfrak{Z}$. Then the following are equivalent:
\begin{enumerate}
\item The group $P(r,n,k,s,q)$ is trivial;
\item $|r-s|=1$ and either $k\equiv 1\bmod n$ or $k\equiv 1+q(r-s) \bmod n$.
\end{enumerate}
\end{corx}
The type $\tilde{\mathfrak{Z}}$ condition assumed in Theorem \ref{theorem:MainResult2} is not a mere convenient restriction as it is satisfied in all known examples where the group $P(r,n,k,s,q)$ is perfect and nontrivial, assuming $(n,k-1,q)=1$ and $r<n$. In fact, the following  conjecture implies that under these conditions, being of type $\tilde{\mathfrak{Z}}$ is a necessary condition for $P(r,n,k,s,q)$ to be a perfect group.

\begin{conj}\label{conj:krforperfect}
Let $n,k,q\geq 1$ and $2\leq r< n$ be integers such that $(n,k-1,q)=1$. Suppose that $k\not\equiv 1\bmod n$ and $k\not\equiv 1+q\bmod n$. If $P(r,n,k,r-1,q)$ is perfect, then it is of type $\tilde{\mathfrak{Z}}$.
\end{conj}
The removed cases  $k\equiv 1\bmod n$ and $k\equiv 1+q\bmod n$ in Conjecture  \ref{conj:krforperfect} give rise to the trivial group by  Lemma \ref{lem:Trivialgroup}. Also if $r\equiv 0\bmod n$ and $(n,q)=1$, then $P(r,n,k,r-1,q)$ is perfect by Lemma \ref{lem:requivn}. (Note that although $(n,q)=1$ is not given as a condition in the statement of Conjecture \ref{conj:krforperfect}, if $P(r,n,k,r-1,q)$ is perfect and $(n,k-1,q)=1$, then $(n,q)=1$ by Lemma \ref{lem:gcd(n,q)dividesk-1}). Finally, Theorem \ref{theorem:MainResult3} can be used to reduce the cases where $r>n$ to the cases $r<n$. Hence, all the removed cases are well-understood. Therefore, a proof of Conjecture \ref{conj:krforperfect}, together with Theorem \ref{theorem:MainResult2}, will yield  the complete classification of perfect groups $P(r,n,k,s,q)$. Note that following the irreducible assumption $(n,k-1,q)=1$, Conjecture \ref{conj:krforperfect}  implies that for even $n\geq 2$ and $r\not\equiv 0\bmod n$, the group $P(r,n,k,r-1,q)$ is perfect if and only if the conditions of Lemma \ref{lem:Trivialgroup} hold, that is, $P(r,n,k,r-1,q)$ is the trivial group. 

\medskip
Our final result shows that to determine if $P(r,n,k,r-1,q)$ is perfect, then it can be assumed that $1\leq r\leq n$. 
\begin{theoremx}\label{theorem:MainResult3}
Let $k,q\geq 1$, $1\leq r\leq n$, $\alpha,\beta\geq 0$  and $n\geq 2$ be integers such that $(n,k-1,q)=1$. Then $P(\alpha n+r,n,k,\beta n+r-1,q)$  is perfect if and only if $\alpha=\beta$ and  $P(r,n,k,r-1,q)$  is perfect.
\end{theoremx}

\subsection{Method of proof}
The proofs presented in this article make use of the theory of circulant matrices. The circulant matrix $C=circ_n(c_0,c_1,\cdots,c_{n-1})$ is the $n\times n$ matrix whose first
row is $c_0,c_1,\cdots,c_{n-1}$ and where each subsequent row is a cyclic shift of its predecessor by one column. Hence, for each $0\leq i<n$, if the exponent sum of the generator $x_i $ in the word $w$ is $c_i$, then the relation matrix of the abelianization $G_n(w)^{ab}$ of $G_n(w)$ is the circulant matrix $C$.
The \emph{representer polynomial} of $C$ (alternatively of $G_n(w)^{ab}$ or $G_n(w)$) is the polynomial
\begin{equation*}
f(t)=\sum_{i=0}^{n-1}c_it^i.
\end{equation*}
It is well-known (see for example \cite[Equation 3.2.14]{MR543191} or \cite[Theorem 3, page 78]{johnson1980topics}) that the absolute value of the determinant  of $C$ is given by $\vert\prod_{\lambda^n=1}f(\lambda)\vert$, where the product is taken over all $n$th roots of unity $\lambda$. Therefore the order of  $G_n(w)^{ab}$, denoted by $|G_n(w)^{ab}|$,  is equal to the absolute value of the determinant of $C$ when it is nonzero, and $G_n(w)^{ab}$  is infinite otherwise. This is a well-known fact (see \cite[Theorem 3, page 78]{johnson1980topics}) that has been commonly used in the study of cyclically presented groups. For the Prishchepov group $P(r,n,k,s,q)$ with the cyclic presentation given in (\ref{eqn:Prischepovpresentation}) but with $\epsilon=-1$, the corresponding representer polynomial is given by:
\begin{equation}\label{eqn:representerpoly}
f(t)=\sum_{i=0}^{r-1}t^{iq}-t^{k-1}\sum_{i=0}^{s-1}t^{iq}.
\end{equation}
Therefore, if  $P(r,n,k,s,q)^{ab}$ is a finite group, then its order is is given by the formula
\begin{equation}\label{order}
|P(r,n,k,s,q)^{ab}|=\Big\vert\prod_{\lambda^n=1}f(\lambda)\Big\vert.
\end{equation} 
It is not hard to see that if the representer polynomials of $G_n(u)$ and $G_n(u')$ are equal to $\sum_{i=0}^{n-1}c_it^i$, say, then $G_n(u)^{ab}\cong G_n(u')^{ab}\cong G_n(x_0^{c_0}x_1^{c_1}\ldots x_{n-1}^{c_{n-1}})^{ab}$. In general, if $G_n(w)$ has representer polynomial $f(t)$, then  $G_n(w)^{ab}$ is isomorphic to the additive group of the quotient of the ring $\mathbb{Z}[t]$ by the ideal $\langle f(t),t^n-1\rangle$  so $G_n(w)$ is perfect (alternatively $\mathbb{Z}[t]/\langle f(t),t^n-1\rangle$ is the zero ring) if and only if $f(t)$ is a unit in the quotient ring $\mathbb{Z}[t]/\langle t^n-1\rangle$ by \cite[Theorem 4, page 78]{johnson1980topics}. This provides an alternative method of proving perfectness which is sometimes easier to use than the formula in (\ref{order}).  In particular we apply it to obtain an alternative proof of the classification of the perfect generalized Sieradski groups $S(r,n)$ (see Section \ref{sec:Generalcase}), and  we expect that it may be easier to prove Conjecture \ref{conj:krforperfect} using this approach. 

\subsection{Structure}
The rest of this article is organised as follows. In Section \ref{sec:Preliminary} we obtain various preliminary results that are used in subsequent sections. In Section \ref{sec:sec4} we prove Theorem \ref{theorem:MainResult1} and deduce Corollary \ref{theorem:LOGgroup}. In Section \ref{sec:Generalcase} we prove Theorem \ref{theorem:MainResult2} and deduce Corollary \ref{cor:ToMainResult2}. In Section \ref{sec:SEC3} we prove Theorem \ref{theorem:MainResult3}. In Section \ref{sec:conclusion} we conclude and pose a question concerning trivial Prishchepov groups.
\section{Preliminaries}\label{sec:Preliminary}
	
This section contains preliminary results that are used in the subsequent sections. In terms of the parameters $r,n,k,s,q$, these results give necessary conditions for $P(r,n,k,s,q)$ to be perfect,  sufficient conditions for $P(r,n,k,s,q)$ to be trivial, and the relationship between the abelianizations of a certain pair of Prishchepov groups.

\begin{lem}\label{lem:r-s}
	Let $k,r,s,q\geq 1$  and $n\geq 2$ be integers and let $\mathcal{P}$ be as in (\ref{eqn:Prischepovpresentation}). If  $\mathcal{P}$ defines a perfect group, then $|r+\epsilon s|=1$, so $\epsilon=-1$.
\end{lem}
\begin{proof}
	By adding the relations $x_i=x_0$ for $i=1,2,\cdots,n-1$ to $\mathcal{P}$, the resulting presentation is equivalent to $\langle x_0~|~x_0^{r+\epsilon s} \rangle $, which defines the cyclic group of order $|r+\epsilon s|$. Therefore if $\mathcal{P}$ defines a perfect group, then $|r+\epsilon s|=1$, so $\epsilon=-1$ since $r,s\geq 1$. 
\end{proof}

When $\epsilon=-1$ the presentation $\mathcal{P}$ in (\ref{eqn:Prischepovpresentation}) defines the Prishchepov groups $P(r,n,k,s,q)$. The next result gives a sufficient condition for $P(r,n,k,s,q)$ to be the trivial group.
\begin{lem}(\cite[Theorem 17(b)]{williams2012largeness})\label{lem:Trivialgroup}
	Let $k,r,s,q\geq 1$  and $n\geq 2$ be integers. If  $|r-s|=1$ and either $k\equiv 1\bmod n$ or $k\equiv 1+q(r-s)\bmod n$, then $P(r,n,k,s,q)$ is trivial. 
\end{lem}

As earlier mentioned to classify the perfect groups $P(r,n,k,s,q)$ it is enough to restrict to the case where $(n,q,k)=1$. With this restriction, it becomes a necessary condition for $P(r,n,k,s,q)$ to be perfect that $k\equiv 1\bmod (n,q)$ as the following result shows.

\begin{lem}\label{lem:gcd(n,q)dividesk-1}
	Let $n,k,r,s,q\geq 1$ be integers. If $P(r,n,k,s,q)$ is perfect, then $k\equiv 1\bmod (n,q)$.
\end{lem}
\begin{proof}
	Let $P(r,n,k,s,q)$ be a perfect group. Then  $|r-s|=1$ by Lemma \ref{lem:r-s}. Let $N$ be any divisor of $(n,q)$. Then the map from the group  $P(r,n,k,s,q)$ to the group $P(r,N,k,s,q)$ sending the generator $x_i$ to the generator $x_{i\bmod N}$ is a surjective group homomorphism. The group $P(r,N,k,s,q)$ has a  presentation $P$ of the form 
	\begin{align*}
	P&=\Bigg\langle x_0, \cdots , x_{N-1} ~\Big |~\prod^{r-1}_{i=0}x_{j+iq}=\prod^{s-1}_{i=0}x_{j+k-1+iq} ~~(0\leq j<N) \Bigg\rangle
	\end{align*}
where the subscripts are taken modulo $N$. As $q\equiv 0\bmod N$, we have that
		\begin{align*}
	P&=\langle x_0,x_1,\cdots,x_{N-1} ~|~ x_j^{r }=x_{j+k-1}^{s} ~~(0\leq j<N)\rangle\\
	&=\langle x_0,x_1,\cdots,x_{N-1} ~|~ x_j^{r }x_{j+k-1}^{-s} ~~(0\leq j<N)\rangle.
	\end{align*}
	Suppose  that  $k-1\not \equiv 0\bmod N$. Then in particular $N>1$. Let $m=N/(N,k-1)>1$.
	Note that $P$ is a presentation of the free product of $(N,k-1)$ copies of the group $G_{m}(x_0^{r}x_{1}^{-s})$. Since $(r,s)=1$, we have $G_{m}(x_0^{r}x_{1}^{-s})\cong \mathbb{Z}_{|r^m-s^m|}$ by  \cite[Lemma 6]{bogley2017coherence}. Hence $P(r,N,k,s,q)$ is not perfect since $|r^m-s^m|>1$. Therefore $k-1 \equiv 0\bmod N$, as required.
\end{proof}
Theorem \ref{theorem:MainResult3}  gives a connection between the abelianizations of $P(r,n,k,s,q)$ and $P(r',n,k,s',q)$ when $r\equiv r'\bmod n$ and $s\equiv s'\bmod n$. We will deduce Theorem \ref{theorem:MainResult3} from Corollary \ref{cor:isomorphicab}. The main ingredient used in the proof of the latter is the following.

\begin{lem}\label{lem:isomorphicab}
Let $n,q,s,s'\geq 1$ be integers. If $s-s'=\alpha n$ where $\alpha>0$, $(n,q)=1$, $\lambda\neq 1$, and $\lambda^n=1$, then 
\[\sum_{i=0}^{s-1}\lambda^{iq}=\sum_{i=0}^{s'-1}\lambda^{iq}.\]
\end{lem} 
\begin{proof}
Assume $s\geq s'$. Since $(n,q)=1$, we have $\lambda^q\neq 1$ so
	\begin{align*}
	\sum_{i=0}^{s-1}\lambda^{iq}&=\sum_{i=0}^{s'-1}\lambda^{iq}+\sum_{i=s'}^{s-1}\lambda^{iq}\\
	&=\sum_{i=0}^{s'-1}\lambda^{iq}+\lambda^{s'q}\sum_{i=0}^{s-s'-1}\lambda^{iq}\\
	&=\sum_{i=0}^{s'-1}\lambda^{iq}+\lambda^{s'q}\sum_{i=0}^{\alpha n-1}\lambda^{iq}\\
	&=\sum_{i=0}^{s'-1}\lambda^{iq}+\lambda^{s'q}(\dfrac{\lambda^{\alpha n q}-1}{\lambda^q-1})\\
	&=\sum_{i=0}^{s'-1}\lambda^{iq}.
	\end{align*}
\end{proof}

\begin{cor}\label{cor:isomorphicab}
	Let $r,s,k,q\geq 1$ and $n\geq 2$. Let $r=\alpha n+r'$ and $s=\beta n+s'$ for some integers $\alpha, \beta>0$ and $1\leq r',s'\leq n$. Suppose $r'\neq s'$, $(n,q)=1$ and $|P(r',n,k,s',q)^{ab}|$ is finite. Then  \[|P(r,n,k,s,q)^{ab}|=
	\Big |\dfrac{(r-s)}{(r'-s')}\Big | \times |P(r',n,k,s',q)^{ab} |.\]
\end{cor}
\begin{proof}
Let $f(t)$ and $f'(t)$ be the representer polynomials of $P(r,n,k,s,q)^{ab}$ and $P(r',n,k,s',q)^{ab}$ respectively. If $\lambda^n=1$ and $\lambda\neq 1$, Then $f(\lambda)=f'(\lambda)$ by Lemma \ref{lem:isomorphicab}, so	we have
	\begin{align*}
	|P(r,n,k,s,q)^{ab}|&=\Big\vert\prod_{\lambda^n=1}f(\lambda)\Big\vert\\
	&=|f(1)|\times \Big\vert\prod_{\lambda^n=1,\lambda\neq 1}f(\lambda)\Big\vert\\
	&=|f(1)|\times \Big\vert\prod_{\lambda^n=1,\lambda\neq 1}f'(\lambda)\Big\vert\\
	&=|f(1)/f'(1)|\times |P(r',n,k,s',q)^{ab}|.
	\end{align*}
	It is easy to show that $f(1)=r-s$ and $f'(1)=r'-s'$, as required. 
\end{proof}

\begin{lem}\label{lem:isoofreppoly}
Let $n,k\geq 2$ be integers. Then $P(r,n,k,r-1,1)^{ab}\cong P(k-1,n,r+1,k-2,1)^{ab}$.
\end{lem}
 \begin{proof}
If $k=2$ or $r=1$ then $P(r,n,k,r-1,1)\cong P(k-1,n,r+1,k-2,1)$ is the trivial group by Lemma \ref{lem:Trivialgroup}. Assume then that $n,k> 2$ and $r>1$. As explained in the introduction  it suffices to show that the representer polynomials of the two groups $P(r,n,k,r-1,1)$ and $P(k-1,n,r+1,k-2,1)$ are the same.
 Suppose $r\leq k$. Then the representer polynomial of the group 
 $P(k-1,n,r+1,k-2,1)$ is \[f(t)=\sum_{i=0}^{k-2}t^i- t^r\sum_{i=0}^{k-3}t^i.\]
 Therefore
 \begin{align*}
 f(t)&=\sum_{i=0}^{k-2}t^i- \sum_{i=0}^{k-3}t^{i+r}\\
 &=\Bigg (\sum_{i=0}^{r-1}t^i+ \sum_{i=r}^{k-2}t^{i}\Bigg ) - \Bigg (\sum_{i=r}^{k-2}t^i+ \sum_{i=k-1}^{k+r-3}t^{i} \Bigg )\\
  &=\sum_{i=0}^{r-1}t^i - \sum_{i=k-1}^{k+r-3}t^{i}\\
  &=\sum_{i=0}^{r-1}t^i - t^{k-1}\sum_{i=0}^{r-2}t^{i}
 \end{align*}
 which is the representer polynomial of the group $P(r,n,k,r-1,1)$. By replacing $k$ with $r+1$ we obtain a proof for the case when $k<r$.
 \end{proof}

We close this section with a result that allows one to assume that $r\not\equiv 0\bmod n$ when proving that $P(r,n,k,s,q)$ is perfect.
\begin{lem}\label{lem:requivn}
Let $k,r,q\geq 1$  and $n\geq 2$ be integers. If $r\equiv 0\bmod n$ and $(n,q)=1$, then the group $P(r,n,k,r-1,q)$ is perfect.
\end{lem}
\begin{proof}
Let $r\equiv 0\bmod n$ and $(n,q)=1$. Then $P(r,n,k,r-1,q)\cong P(r,n,K,r-1,1)$, $1\leq K\leq n$, where $K-1\equiv \hat{q}(k-1)\bmod n$, $q\hat{q}\equiv 1\bmod n$. If $K=1$ or $n=2$, then $P(r,n,K,r-1,1)$ is the trivial group by Lemma \ref{lem:Trivialgroup}. Assume $n,K\geq 2$. Then, by Lemma \ref{lem:isoofreppoly} we have that $P(r,n,K,r-1,1)^{ab}\cong P(K-1,n,r+1,K-2,1)^{ab}$. The result follows since the group $P(K-1,n,r+1,K-2,1)$ is isomorphic to the trivial group $P(K-1,n,1,K-2,1)$.
\end{proof}
 \section{The groups  $H(r,n,s)$}\label{sec:sec4}

In this section is to give a proof of Theorem \ref{theorem:MainResult1} and deduce Corollary  \ref{theorem:LOGgroup}. To do so, we first do some calculations involving the representer polynomial $f(t)$ of $P(r,n,k,s,q)$ evaluated at an $n$th root of unity. Recall from (\ref{eqn:representerpoly}) that $f(t)=F(t)[1-t^{k-1}]+\sum_{i=s}^{r-1}t^{iq}$ where 
\begin{equation}\label{eqn:FOfT}
F(t)=\sum_{i=0}^{s-1}t^{iq},
\end{equation} 
and where it is assumed that $r\geq s$. Let $\bar{z}$ denote the complex conjugate of $z\in \mathbb{C}$ and  let 
\begin{align}\label{six}
\mathcal{F}(t)&=\left\{
  \begin{array}{@{}ll@{}}
    f(t)-1, & \text{if}\ t\in \mathbb{R}\\ 
    f(t)f(\bar{t})-1, & \text{if}  \ t\in \mathbb{C}-\mathbb{R}\\
\end{array}\right.
\end{align}
If $r=s+1$ and $\lambda\in \mathbb{C}-\mathbb{R}$ is an $n$th root of unity, then it can be shown (see proof of Theorem \ref{lem:MlambdaNlambda}) that 
\begin{equation}\label{M}
\mathcal{F}(\lambda) = 2F(\lambda)F(\bar{\lambda})[1-\Re({\lambda}^{k-1})]+2\Re(\bar{\lambda}^{q} F(\bar{\lambda})[1-{\lambda}^{k-1}])
\end{equation}
where $\Re(z)$ means the real part of $z\in \mathbb{C}$.

\medskip
We are now in the position to give our next result, which is the main ingredient used in the proof of Theorem \ref{theorem:MainResult1}.
\begin{theorem}\label{lem:MlambdaNlambda}
	Let $n\geq 2$, $k,q\geq 1$ and $r\geq s\geq 1$ be integers and let $\mathcal{F}(t)$ be as in (\ref{six}). If $P(r,n,k,s,q)$ is perfect,  then  $r=s+1$ and either $|\mathcal{F}(\lambda)+1|< 1$ for some $n$th root of unity $\lambda$, or $|\mathcal{F}(\lambda)+1|= 1$ for all $n$th roots of unity $\lambda$.
\end{theorem}

\begin{proof}
Let $f(t)$ be the representer polynomial of $P(r,n,k,s,q)$. Since $f(t)$ has real coefficients, its nonreal roots come in complex conjugate pairs. Let $\mathcal{S}(n)$ denote a set whose elements are $z$, for each distinct complex conjugate pair $\{z,\bar{z}\}$ of $n$th roots of unity. (For example, $\mathcal{S}(2)=\{-1,1\}$ and  $\mathcal{S}(4)=\{-1,1,\pm i\}$). Let $P(r,n,k,s,q)$ be a perfect group, so $P(r,n,k,s,q)^{ab}$ is a finite. 
For each $\lambda\in \mathcal{S}(n)$, set $\mathcal{L}({\lambda})=f(\lambda)$ if $\lambda$ is real and $\mathcal{L}({\lambda})=f(\lambda)f(\bar{\lambda})$ otherwise.  Then by (\ref{order}) we have
 \begin{equation}\label{eqn:AbelianOrder}
 |P(r,n,k,s,q)^{ab}|=\prod_{{\lambda}\in \mathcal{S}(n)}\mathcal{L}({\lambda}).
 \end{equation}
Note that $r=s+1$  by Lemma \ref{lem:r-s} since $P(r,n,k,s,q)$ is perfect. Hence,  $f(t)=F(t)(1-t^{k-1})+t^{sq}$ where $F(t)$ is as in (\ref{eqn:FOfT}). If $\lambda$ is nonreal, then 
	\begin{align*}
	\mathcal{L}({\lambda})&=f(\lambda)f(\bar{\lambda})=(F(\lambda)-\lambda^{k-1} F(\lambda)+\lambda^{sq})(F(\bar{\lambda})-\bar{\lambda}^{k-1} F(\bar{\lambda})+\bar{\lambda}^{sq})\\
	&=F(\lambda)F(\bar{\lambda})
	-\bar{\lambda}^{k-1}F(\lambda)F(\bar{\lambda})
	+\bar{\lambda}^{sq}F(\lambda)
	-\lambda^{k-1} F(\lambda)F(\bar{\lambda})
	+F(\lambda)F(\bar{\lambda})\\
	&\quad \quad -\lambda^{k-1}\bar{\lambda}^{sq} F(\lambda)
	+\lambda^{sq}F(\bar{\lambda})
	-\bar{\lambda}^{k-1}{\lambda}^{sq} F(\bar{\lambda})
	+1.
	\end{align*}
	Note that 
	\begin{align*}
	\lambda^{sq}F(\bar{\lambda})&=\lambda^q\sum_{i=0}^{s-1} \lambda^{(s-1)q}\bar{\lambda}^{iq}\\
	&=\lambda^q\sum_{i=0}^{s-1} \lambda^{(s-1-i)q}\\
	&=\lambda^q F(\lambda).
	\end{align*} 
	Similarly we can show that $\bar{\lambda}^{sq}F(\lambda)=\bar{\lambda}^q F(\bar{\lambda})$. Hence
	\begin{align*}
	\mathcal{L}({\lambda})
	&=F(\lambda)F(\bar{\lambda})
	-\bar{\lambda}^{k-1}F(\lambda)F(\bar{\lambda})
	+\bar{\lambda}^{sq}F(\lambda)
	-\lambda^{k-1} F(\lambda)F(\bar{\lambda})
	+F(\lambda)F(\bar{\lambda})\\
	&\quad \quad -\lambda^{k-1}\bar{\lambda}^{sq} F(\lambda)
	+\lambda^{sq}F(\bar{\lambda})
	-\bar{\lambda}^{k-1}{\lambda}^{sq} F(\bar{\lambda})
	+1\\
	&=2F(\lambda)F(\bar{\lambda})
	-\bar{\lambda}^{k-1}F(\lambda)F(\bar{\lambda})
	+\bar{\lambda}^{q}F(\bar{\lambda})
	-\lambda^{k-1} F(\lambda)F(\bar{\lambda})
	-\lambda^{k-1}\bar{\lambda}^{q} F(\bar{\lambda})\\
	& \quad \quad  +\lambda^{q}F({\lambda})
	-\bar{\lambda}^{k-1}{\lambda}^{q} F({\lambda})
	+1\\
	&=F(\lambda)F(\bar{\lambda})[2-{\lambda}^{k-1}-\bar{\lambda}^{k-1}]
	+\bar{\lambda}^{q} F(\bar{\lambda})[1-{\lambda}^{k-1}]
	+{\lambda}^{q} F({\lambda})[1-\bar{\lambda}^{k-1}]
	+1\\
	&=2F(\lambda)F(\bar{\lambda})[1-\Re(\lambda^{k-1})]
	+2\Re(\bar{\lambda}^{q} F(\bar{\lambda})[1-{\lambda}^{k-1}])
	+1\\
	&=\mathcal{F}(\lambda)+1.
	\end{align*}
It then follows from (\ref{six}) that $\mathcal{L}(\lambda)=|\mathcal{F}(\lambda)+1|$ for each $n$th root of unity $\lambda$. Hence by (\ref{eqn:AbelianOrder}) we have \[|P(r,n,k,s,q)^{ab}|=\prod_{{\lambda}\in \mathcal{S}(n)}|\mathcal{F}({\lambda})+1|.\]
Therefore, $P(r,n,k,s,q)^{ab}= 1$ implies that either $|\mathcal{F}(\lambda)+1|<1$ for some $n$th root of unity $\lambda$ or $|\mathcal{F}(\lambda)+1|=1$ for all $n$th roots of unity $\lambda$, as required.
\end{proof}

By restricting to $H(r,n,r-1)=P(r,n,r+1,r-1,1)$, equations (\ref{eqn:FOfT}) and (\ref{M}) become $F(t)=1+t+t^2+\cdots+t^{r-2}$ and $\mathcal{F}(\lambda) = 2F(\lambda)F(\bar{\lambda})[1-\Re({\lambda}^{r})]+2\Re(\bar{\lambda} F(\bar{\lambda})[1-{\lambda}^{r}])$ respectively. Our next result shows that  $2\Re(\bar{\lambda} F(\bar{\lambda})[1-{\lambda}^{r}])=0$, so  $\mathcal{F}(\lambda) = 2F(\lambda)F(\bar{\lambda})[1-\Re({\lambda}^{r})]$.
\begin{lem}\label{lem:Mlambdaequals0}
Let  $F(t)=1+t+t^2+\cdots+t^{r-2}$ and let  $z\in \mathbb{C}$  such that $|z|=1$. Then $\Re({z} F({z})[1-\bar{z}^{r}])=0$.
\end{lem}
\begin{proof}
Notice  that $2\Re(z F(z)[1-\bar{z}^r])={z} F({z})[1-\bar{z}^r] +\bar{z} F(\bar{z})[1-{z}^r]$, and
	\begin{align*}
	z F({z})[1-\bar{z}^r] +\bar{z} F(\bar{z})[1-{z}^r]
	&=z F({z})-\bar{z}\bar{z}^{r-2} F({z}) +\bar{z} F(\bar{z})-z{z}^{r-2} F(\bar{z})\\
	&=z F(z)-\bar{z} F(\bar{z}) +\bar{z} F(\bar{z})-z F({z})\\
	&=0.
	\end{align*}
 Therefore $\Re({z} F({z})[1-\bar{z}^{r}])=0$, as required.
\end{proof}

We can now give a proof of Theorem  \ref{theorem:MainResult1}.

\begin{proof}[Proof of Theorem \ref{theorem:MainResult1}]\label{hrns}
$(1)\implies (3)$. Let $H(r,n,s)$ be a perfect group. Since $H(r,n,s)\cong H(s,n,r)$ we assume that $r\geq s$. Hence, $r=s+1$ by Lemma \ref{lem:r-s}. If $n=2$, then there is nothing left to show since either $s$ is even or $r$ is even. Assume $n>2$ and let $\lambda$ be a nonreal  $n$th root of unity. Let $F(t)=1+t+t^2+\cdots +t^{s-1}$ and let  $\mathcal{F}(\lambda)$ be as in (\ref{M}) but with $k=r+1$ and $q=1$. Then by Lemma \ref{lem:Mlambdaequals0} we have \[\mathcal{F}(\lambda)=2F(\lambda)F(\bar{\lambda})[1-\Re(\lambda^{r})].\] By Theorem \ref{lem:MlambdaNlambda} we must have that  $|\mathcal{F}(\lambda)+1|\leq 1$. But $\mathcal{F}(\lambda)$ is a nonnegative real number. Hence $\mathcal{F}(\lambda)=0$, so either $\lambda^{r}=1$ or $F(\lambda)=0$. But $F(\lambda)=(\lambda^s-1)/(\lambda-1)$. So $F(\lambda)=0$ implies $\lambda^s=1$. Therefore, choosing  $\lambda$ to be a primitive $n$th root of unity shows that either $r\equiv 0\bmod n$ or $s\equiv 0\bmod n$, as required.

\medskip
$(3)\implies (2)$. Let $|r-s|=1$ and either $r\equiv 0\bmod n$ or $s\equiv 0\bmod n$. Then $H(r,n,s)$ is the trivial group by Lemma \ref{lem:Trivialgroup}.

\medskip
$(2)\implies (1)$. A trivial group is perfect.
\end{proof}
Conjecture  \ref{conj:WilliamsConjecture} follows from Theorem \ref{theorem:MainResult1}.

\begin{cor}\label{Cor3.3}
Let $k,r,s,q\geq 1$  and $n\geq 2$. If  $k\equiv qr+1$, then $P(r,n,k,s,q)$  is perfect if and only if $|r-s|=1$ and either $k\equiv 1\bmod n$ or $k\equiv 1+q(r-s)\bmod n$.
\end{cor}
\begin{proof}
If the conditions in the statement of the corollary hold, then $P(r,n,k,s,q)$ is trivial by Lemma \ref{lem:Trivialgroup}, and hence perfect.	
	
\medskip
Suppose $P(r,n,k,s,q)$  is perfect, where $k=qr+1$. Let $N=n/(n,q)$. Then $P(r,n,k,s,q)$ is the free product of $(n,q)$ copies of $H(r,N,s)\cong H(s,N,r)$.  By Theorem \ref{theorem:MainResult1} we have $|r-s|=1$, and  either $r\equiv 0\bmod N$ or $s\equiv 0\bmod N$. Therefore, either $qr\equiv 0\bmod n$  or $qs\equiv 0\bmod n$. As $k\equiv qr+1\bmod n$, if   $qr\equiv 0\bmod n$, then $k\equiv 1\bmod n$. Suppose $qs\equiv 0\bmod n$. If $s=r+ 1$, then $q(r+1)\equiv 0\bmod n$, so  $k\equiv 1-q\bmod n$. Similarly, if $s=r- 1$, then $q(r-1)\equiv 0\bmod n$, so  $k\equiv 1+q\bmod n$, as required.
\end{proof}
We now prove Corollary  \ref{theorem:LOGgroup}.
\begin{proof}[Proof of Corollary  \ref{theorem:LOGgroup}]
It suffices by \cite[Theorem A]{williams2019generalized} to show that if $r\not\equiv 0\bmod n$ and $s\not\equiv 0\bmod n$, then the group $H(r/2,n/2,s/2)$ is not perfect. Suppose for contradiction that these conditions hold but $H(r/2,n/2,s/2)$ is perfect. Then in particular either $r/2\equiv 0\bmod n/2$ or $s/2\equiv 0\bmod n/2$ by Theorem \ref{theorem:MainResult1}, so  either $r\equiv 0\bmod n$ or  $s\equiv 0\bmod n$ which contradicts the assumptions above.
\end{proof}

\section{Groups of type $\tilde{\mathfrak{Z}}$}\label{sec:Generalcase}
In this section we prove Theorem \ref{theorem:MainResult2} by showing that the perfect Prishchepov groups of type $\tilde{\mathfrak{Z}}$ are closely related to the generalized Sieradski groups. But first, we mention some relevant results about the generalized Sieradski groups. 
\begin{theorem}(\cite[Corollary 3.2]{cavicchioli1998geometric}) \label{inf}
For $r\geq 2$, the generalized Sieradski group $S(r,n)$ is infinite if and only if $n\geq (4r-2)/(2r-3)$.
\end{theorem}

By the main result of \cite{cavicchioli1998geometric}, $S(r,n)$ corresponds to a spine of the $n$-fold cyclic covering of the $3$-sphere branched over the torus knot $K(2r-1,2)$, that is, the Brieskorn manifold $M(n,2r-1,2)$ in the sense of \cite{MR0418127}. By \cite{MR206972,MR1555366}, the manifold $M(n,2r-1,2)$ is an integer homology sphere  if and only if $(n,4r-2)=1$. It then follows from this  that $S(r,n), r\geq 2$ is perfect if and only if $(4r-2,n)=1$. An alternative but elementary proof of this result is given below.
 \begin{lem}\label{lem:perfectSieradski}
Let $n,r\geq 2$. Then $S(r,n)$ is perfect if and only if $(4r-2,n)=1$.
\end{lem}
\begin{proof}
As mentioned in Section \ref{sec:Introdution} it is enough to show that the representer polynomial $f(t)$ of $S(r,n)^{ab}$, that is,
\begin{align*}
f(t)&=\sum_{i=0}^{r-1}t^{2i}-t\sum_{i=0}^{r-2}t^{2i}\\
&=(t^{2r-1}+1)/(t+1)
\end{align*}
 is a unit in the ring $\mathbb{Z}[t]/\langle t^n-1\rangle$. Note that $f(t)$ is a unit if and only if the quotient $R$ of $\mathbb{Z}[t]/\langle t^n-1\rangle$ by the ideal generated by $(1+t)f(t)$ is the zero ring. But $(1+t)f(t)=t^{2r-1}+1$, so in $R$, $t^{4r-2}=1$. Hence $R$ is the zero ring if and only if $(4r-2,n)=1$, as required.
\end{proof}

We now give a proof of Theorem \ref{theorem:MainResult2}. Recall that  $P(r,n,k,s,q)$ is of type $\tilde{\mathfrak{Z}}$ if either $2(k-1)\equiv q(r-s)\bmod n$, that is, it is  of type $\mathfrak{Z}$;  or $q(r+s)\equiv 0 \bmod n$, that is, it is of type $\mathfrak{Z}'$.

\begin{proof}[Proof of Theorem \ref{theorem:MainResult2}]
$(1)\implies (2)$. Let  $P(r,n,k,s,q)$ be a perfect group. Then  by Lemma \ref{lem:r-s} we have $|r-s|=1$. Since $r\geq s$ by assumption, we have $s=r-1$.  If $(n,q)>1$, then $(n,q)$ does not divide $k-1$ since $(n,q,k-1)=1$ by assumption, so $P(r,n,k,r-1,q)$ is not perfect by Lemma \ref{lem:gcd(n,q)dividesk-1}. Hence $(n,q)=1$.
Let $(n,k-1-qr)=N$. Then there is an epimorphism from $P(r,n,k,r-1,q)$ to $P(r,N,qr+1,r-1,q)$ sending $x_i$ to $x_{i\bmod N}$ $(0\leq i<n)$. Since $(n,q)=1$, the group $P(r,N,qr+1,r-1,q)$ is isomorphic to $H(r,N,r-1)$. Hence by Theorem \ref{theorem:MainResult1}, $P(r,N,qr+1,r-1,q)$ is perfect if and only if $r\equiv 0\bmod N$ or $r\equiv 1\bmod N$. 

\medskip
The type $\tilde{\mathfrak{Z}}$ condition implies that $2r\equiv 1\bmod n$ or $2(k-1)\equiv q\bmod n.$
Suppose first that $2r\equiv 1\bmod n$. Then $2r\equiv 1\bmod N$. Since $r\equiv 0\bmod N$ or  $r\equiv 1\bmod N$, we conclude that $N=1$. Also if the condition $2(k-1)\equiv q\bmod n$ holds, then   $2(k-1)\equiv q\bmod N$. But by assumption  $k-1\equiv qr\bmod N$. Hence  $q(2r-1)\equiv 0\bmod N$. Since $(n,q)=1$, we have that $2r-1\equiv 0\bmod N$. Therefore as before we obtain that  $N=1$. 

\medskip
$(2)\implies (1).$ Let the conditions $(n,q)=1$, $|r-s|=1$ and $(n,k-1-qr)=1$ hold. Suppose first that $2(k-1)\equiv q\bmod n$. Then $(n,q)=(n,2(k-1))=(n,k-1)=1$. So $n$ is odd and 
\begin{align*}
P(r,n,k,r-1,q)&\cong P(r,n,k,r-1,2(k-1))\\
&\cong P(r,n,2,r-1,2)\\
&\cong S(r,n).
\end{align*}
Hence, as $n$ is odd it suffices to show that $(n,2r-1)=1$  by Lemma \ref{lem:perfectSieradski}. This follows from the assumption $(k-1-qr,n)=1$ for the following reason. The condition $(n,2r-1)=1$ holds if and only if  $(n,(2r-1)(k-1))=1$ since $(k-1,n)=1$. But $(n,(2r-1)(k-1))=(n,k-1-qr)$, as required.

\medskip
On the other hand, suppose that $q(2r-1)\equiv 0\bmod n$. Since $(n,q)=1$ we have that $2r\equiv 1\bmod n$. As before we conclude that $n$ is odd. Also, the assumption $(n,q)=1$ implies that there exists integers $1\leq p,\bar{q}\leq n$ such that $\bar{q}q\equiv 1\bmod n$ and $\bar{q}(k-1)\equiv p\bmod n$. Hence $P(r,n,k,r-1,q)\cong P(r,n,p+1,r-1,1)$. Since $k-1\not\equiv 0,q\bmod n$ by assumption, we have that $2\leq p< n$. Hence 
$P(r,n,p+1,r-1,1)^{ab}\cong P(p,n,r+1,p-1,1)^{ab}$ by Lemma \ref{lem:isoofreppoly}. Therefore, we have
\begin{align*}
P(r,n,k,r-1,q)^{ab}&\cong P(r,n,p+1,r-1,1)^{ab}\\
&\cong P(p,n,r+1,p-1,1)^{ab}\\
&\cong P(p,n,r+1,p-1,2r)^{ab}\\
& \cong P(p,n,2,p-1,2)^{ab}\\
& \cong S(p,n)^{ab}.
\end{align*}
As $n$ is odd it suffices to show that $(n, 2p-1)=1$  by Lemma \ref{lem:perfectSieradski}. This follows from the assumption $(k-1-qr,n)=1$. To see why this is true notice that $(n,2p-1)=1$ if and only if $(n,qr(2p-1))=1$ since $2r\equiv 1\bmod n$ and $(n,q)=1$. Using $\bar{q}(k-1)\equiv p\bmod n$, we have that   $(n,qr(2p-1))=(n,2r(k-1)-qr)=(n,k-1-qr)$, as required. 
\end{proof}
We now prove Corollary \ref{cor:ToMainResult2}

\begin{proof}[Proof of Corollary \ref{cor:ToMainResult2}]
$(1)\implies (2).$ Let $P(r,n,k,s,q)$ be the trivial group. Then $s=r-1$ by Lemma \ref{lem:r-s}. 
Since $P(r,n,k,s,q)$ is  of type $\mathfrak{Z}$  by assumption, it follows from the proof of Theorem \ref{theorem:MainResult2} that $P(r,n,k,s,q)\cong S(r,n)$. First suppose that $n\geq 3$. Then by Lemma \ref{lem:perfectSieradski}  $P(r,n,k,s,q)$ is infinite if $r\geq 3$ or $n\geq 6$. Hence we can assume that $r=2$ and $n=3,4,5$. These give rise to finite groups of order $8,24,120$ respectively. Therefore $n=2$. Hence, either $k\equiv 1\bmod n$ or $k\equiv 2\bmod n$, as required.

\medskip
$(2)\implies (1).$ This is by Lemma \ref{lem:Trivialgroup}.
\end{proof}
\section{On Theorem \ref{theorem:MainResult3}}\label{sec:SEC3}
We deduce Theorem \ref{theorem:MainResult3}  from  Corollary \ref{cor:isomorphicab} as follows.
\begin{proof}[Proof of Theorem \ref{theorem:MainResult3}]
Let $P(\alpha n+r,n,k,\beta n+r-1,q)$ be a perfect group and let $(n,q,k-1)=1$. By Lemma \ref{lem:gcd(n,q)dividesk-1} we have $k-1\equiv 0\bmod (n,q)$. It then follows from  the assumption $(n,q,k-1)=1$ that $(n,q)=1$. Hence by Corollary \ref{cor:isomorphicab},  we have 
\[|P(\alpha n+r,n,k,\beta n+r-1,q)^{ab}|=|(\alpha-\beta)n+1|\times |P(r,n,k,r-1,q)^{ab}|.\]
Therefore, $\alpha=\beta$ and $P(r,n,k,r-1,q)$ is perfect. The converse direction is straightforward.
\end{proof}

The following classification of perfect groups $G_n(q, k-1)=P(2,n,k,1,q)$ provides some support for  Conjecture \ref{conj:krforperfect}.
\begin{theorem}[\cite{odoni1999some,williams2010unimodular}]\label{theorem:Willuni}
Suppose $(n, q, k-1) = 1$. Then $G_n(q, k-1)$ is perfect if and only if ($(n,6)=1$ and $q\equiv 2(k-1)\bmod n$) or $k\equiv 1\bmod n$ or $k\equiv 1+ q\bmod n$.
\end{theorem}
By combining Theorem \ref{theorem:Willuni} and  Lemma \ref{lem:isoofreppoly} we deduce the following.
 \begin{cor}
Let $n\geq 2$ and $r\geq 1$ be integers. The group $P(r,n,3,r-1,1)$ is perfect if and only if $(6,n)=1$ and $2r\equiv 1\bmod n$, or $r\equiv 0\bmod n$, or $r\equiv 1\bmod n$.
\end{cor}
 \begin{proof}
 If $r=1$, then $P(r,n,3,r-1,1)$ is trivial, so perfect. Hence, we assume that $r\geq 2$. By Lemma \ref{lem:isoofreppoly} the group $P(r,n,3,r-1,1)$ is perfect if and only if $P(2,n,r+1,1,1)$ is perfect.  The result follows by Theorem \ref{theorem:Willuni}.
 \end{proof}
The next result extends Theorem \ref{theorem:Willuni} to the case where $r\equiv 2\bmod n$, and  classifies the trivial groups in this setting. 
\begin{cor}\label{cortheorem:last}
Let $k,q\geq 1$, $\alpha,\beta\geq 0$  and $n\geq 2$ be integers such that $(n,k-1,q)=1$. Then
\begin{enumerate}
\item $P(\alpha n+2,n,k,\beta n+1,q)$  is perfect if and only if $\alpha=\beta$; and (i) $k \equiv 1\bmod n$; or (ii) $k\equiv 1+ q\bmod n$ or (iii) $((n,6)=1$ and $q\equiv 2(k-1) \bmod n)$.
\item $P(\alpha n+2,n,k,\beta n+1,q)$  is trivial if and only if $\alpha=\beta$ and either $k \equiv 1\bmod n$ or $k\equiv 1+ q\bmod n$.
\end{enumerate}
\end{cor}
\begin{proof}
Part 1 is immediate from a combination of Theorem \ref{theorem:MainResult3} and Theorem \ref{theorem:Willuni}. For part 2, suppose that $P(\alpha n+2,n,k,\beta n+1,q)$ is the trivial group. Then by part 1 we have that $\alpha=\beta$ and either $k \equiv 1\bmod n$ or $k\equiv 1+ q\bmod n$,  so the group is trivial by Lemma \ref{lem:Trivialgroup};  or  $((n,6)=1$ and $q\equiv 2(k-1))$. In the latter case the group is  isomorphic to $S(\alpha n+2,n)=P(\alpha n+2,n,2,\alpha n+1,2)$ as $(n,q,k-1)=1$ by assumption.  
Since $(n,6)=1$, it follows that $n\geq 5$. Therefore, by Theorem \ref{inf} if $S(\alpha n+2,n)$ is finite, then $\alpha=0$ and $n=5$, that is, the group $S(2,5)$, which is not the trivial group.
\end{proof}

\section{Conclusion}\label{sec:conclusion}
In this article, we dealt with the question of when cyclically presented groups of type $\mathfrak{F}$ are perfect or trivial. We showed that to answer this question it is enough to restrict to the sub-family of Prishchepov groups. In terms of perfectness, a complete classification was given for the Prishchepov groups of type $\tilde{\mathfrak{Z}}$ and Conjecture \ref{conj:krforperfect} implies that this is sufficient. In other words, a proof of Conjecture \ref{conj:krforperfect}, together with Theorem \ref{theorem:MainResult2}, will yield the classification of perfect cyclically presented groups of type $\mathfrak{F}$.

\medskip
Some evidence was provided in support of Conjecture \ref{conj:krforperfect}. Also, we have and verified the conjecture computationally for $n\leq 340$. Except for the trivial groups $P(r,n,k,s,q)$ which occur when the conditions of Lemma \ref{lem:Trivialgroup} hold, the pieces of evidence in this article supporting Conjecture \ref{conj:krforperfect} imply that if $(n,k-1,q)=1$ and $P(r,n,k,s,q)$ is the trivial group, then the condition $2r\equiv 1\bmod n$ must hold. Hence it is natural to ask the following question.
\begin{quest}\label{quest:LAST}
Let $2<k\leq n$ and $r\geq 4$ be integers. Is it true that  $P(r,n,k,r-1,1)$ is not the trivial group?
\end{quest}
Note that  Question \ref{quest:LAST} has a negative answer if we allow $r=3$. Indeed, it can be checked in GAP \cite{GAP4} that the groups $P(3,5,3,2,1)$ and $P(3,5,5,2,1)$ are trivial. These groups are the only counter-examples we know that stop the conditions in Lemma \ref{lem:Trivialgroup} from being necessary for $P(r,n,k,s,q)$ to be the trivial  group, assuming $(n,k-1,q)=1$. We expect an affirmative answer to Question \ref{quest:LAST}, and this  is supported by results in this article, and also by the list of trivial groups given in \cite{MR2658419,MR3223773}. Since every  perfect (in particular trivial) group $P(r,n,k,r-1,q)$, such that $(n,k-1,q)=1$ holds, is isomorphic to one of the form $P(r,n,K,r-1,1)$, given Theorem \ref{theorem:MainResult2}, a proof of Conjecture \ref{conj:krforperfect} together with an affirmative answer to Question \ref{quest:LAST}, will yield a complete classification of trivial groups of type $\mathfrak{F}$.

\section*{Acknowledgements}

The first author's research is supported by the Leverhulme Trust Research Project Grant RPG-2017-334. The first author thanks Chimere Anabanti for the encouraging discussions at the onset of this work. Both authors thank Jim Howie and Gerald Williams for reading through earlier drafts of this manuscript and making useful suggestions, and the referee for insightful observations that led to improvements in the article. 

\bibliographystyle{plain}
\bibliography{ReferencesforPerfectPrishchepovGroups}
\end{document}